\documentclass[10pt,oneside]{amsart}
\usepackage{amssymb, amscd, amsmath, amsthm, latexsym, hyperref,}
\usepackage{pb-diagram, fancyhdr, graphicx, psfrag, enumerate, color}
\usepackage[text={6.3in,8.5in},centering,letterpaper,dvips]{geometry}
\def\Z{{\mathbb Z}}

\def\Q{{\mathbb Q}}

\def\mo{\mathcal{M}_F}
\def\orho{\overline{\rho}}

\def\cs{\mathbin{\#}}
 
\newcommand{\fig}[2] { \includegraphics[scale=#1]{#2} }

\newtheorem{theorem}{Theorem}

\newtheorem{corollary}[theorem]{Corollary}

\theoremstyle{definition}

\newtheorem{example}[theorem]{Example}

\def\co{\colon\thinspace}
 
\begin{document}
\title{Chiral smoothings of knots}

\author{Charles Livingston}
\thanks{This work was supported by a grant from the National Science Foundation, NSF-DMS-1505586.}
\address{Charles Livingston: Department of Mathematics, Indiana University, Bloomington, IN 47405 }

\email{livingst@indiana.edu}


\begin{abstract} Can smoothing a single crossing in a diagram for a knot convert it into a diagram of the knot's mirror image?  Zekovi\'c found such a smoothing for the torus knot $T(2,5)$, and Moore-Vasquez proved that such smoothings do not exist for other torus knots $T(2,m)$ with $m$ square free. The existence of such a smoothing implies that $K \cs K$ bounds a Mobius band in $B^4$.  Casson-Gordon theory then provides new obstructions to the existence of such chiral smoothings, in particular removing the constraint that $m$ be square free in the Moore-Vasquez theorem, with the exception of $m=9$, which remains an open case.  Heegaard Floer theory provides further obstructions, which do not apply to the case of $T(2,m)$  but do give constraints in the case of $T(2k,m)$ for $k >1$. Similar methods can be applied to provide lower bounds on the required  number of smoothings needed to convert a knot $K$ into a knot $J$.

\end{abstract}

\maketitle


\section{Introduction}\label{sec:intro}

Given a knot $K\subset S^3$, does it have a diagram for which smoothing a single crossing results in a diagram of its mirror image?   If so, the knot is said to {\it support a chiral smoothing}.  The existence of such a smoothing is easily seen to be equivalent to the existence of   a {\it chiral band move}~\cite{MR3307286,2016arXiv160201542I} or a chiral $H(2)$ move~\cite{MR1075165}, and such a move is called {\it chirally cosmetic}.   The problem naturally generalizes, asking if knots $K$ and $J$ are related by a single smoothing.  Literature on this topic, including its relationship to problems in the study of DNA, 
includes~\cite{MR3307286,MR1075165,2017arXiv170700152I,2016arXiv160201542I,MR2755489,MR2812265,MR3548474,MR2573402}.  Perhaps the starting point of this line of research was in the work of Lickorish~\cite{MR859958}, asking whether given knots could by unknotted with a single band move.

A  basic example, first  discovered by  Zekovi\'c~\cite{MR3331990}, is that  the  torus knot $T(2,5)$ supports a chiral smoothing.   Figure~\ref{fig:2-5} illustrates such a smoothing.  In the reverse direction, Moore and Vazquez~\cite{2018arXiv180602440M} showed that $T(2,5)$ is unique among positive torus knots $T(2,m)$, with $m$  square-free, for which such a move exists. 

\begin{figure}[h]
\fig{.10}{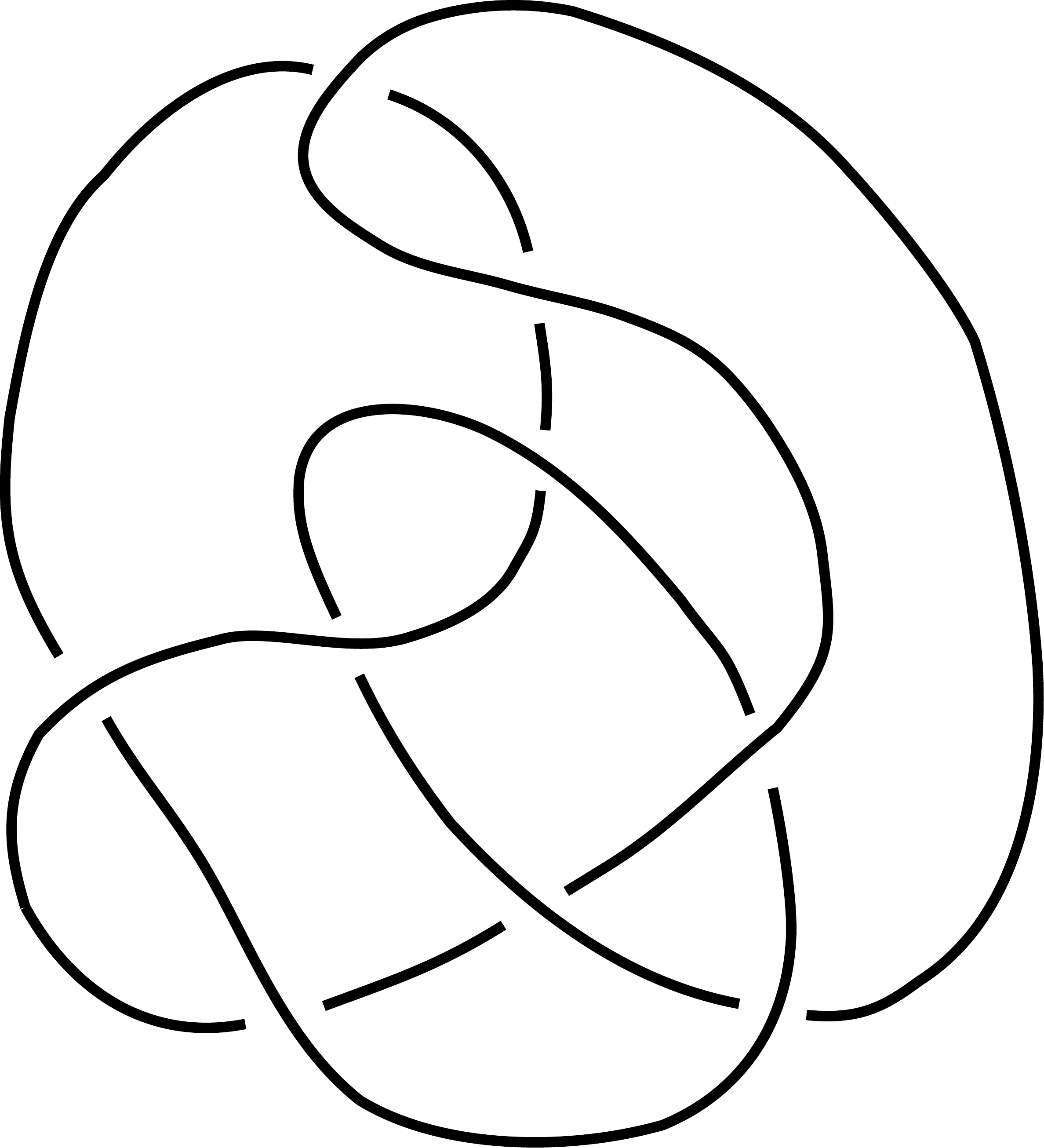}\hskip1in \fig{.10}{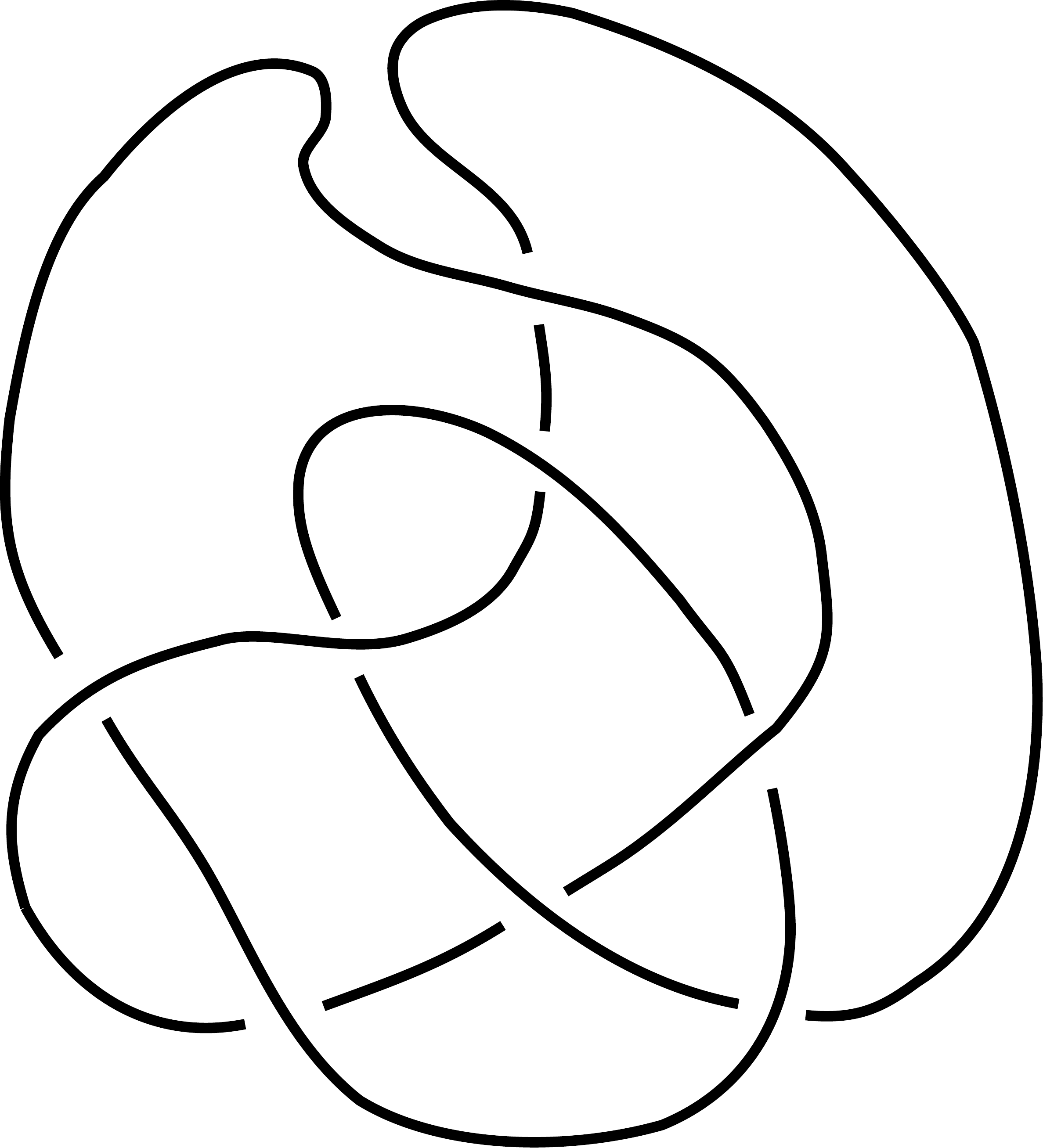}\\
\caption{$T(2,-5)$ smoothed to $T(2,5)$}
\label{fig:2-5}
\end{figure}

If $K$ supports a chiral smoothing, then  a single band move converts  $K\cs K$ into $K\cs -K$, where $-K$ denotes the mirror image of $K$ with string orientation reversed.   The knot $-K$ is the concordance inverse of $K$, meaning that $K \cs -K$ is a slice knot, but in fact it  is  a ribbon knot.  Thus, if $K$ supports a chiral smoothing, then $K\cs K$  bounds a ribbon  Mobius band in $B^4$.   It follows that Casson-Gordon theory~\cite{MR900252,MR656619} can be applied.  We will follow this approach by using the application of Casson-Gordon theory to nonorientable surfaces in $B^4$ that was developed in~\cite{MR2855790}.  One corollary is that the condition that $m$ be square-free can be removed from the Moore-Vasquez result, with the exception of $T(2,9)$, which remains an unknown case.

We will concentrate on the case of $T(2,m)$, since these knots have two-fold branched covers that are lens spaces, $L(2m+1, 1)$;   Casson-Gordon invariants for these spaces can be computed in closed form.  
This work can be extended in a number of ways:
\begin{itemize}
\item Casson-Gordon invariants of two-bridge knots, $B(s,q)$, can be easily computed, lending themselves to produce further examples.  
\item Pairs of knots $(K,J)$ rather than $(K,-K)$ can be considered.  This is related to the $H(2)$ distance, studied, for instance, by Kanenobu in~\cite{MR2812265}.
\item The obstruction we focus on is based on the maximum of the absolute values of a set of sums, $ \{|a_i + a_j|\}$, where the $(a_i,a_j)$ are specified pairs taken from a set $\{a_i\}$.  We use a weak bound on these sums, based on constraints on the set,  given by $ |\min\{a_i\} + \max\{a_i \}|$; this  can be considerably improved.
\end{itemize}
We will provide a few examples of such extensions;   a more thorough study will be left to the future.

In the final section, we briefly discuss the application of Heegaard Floer techniques and results of Ozsv\'ath-Stipsicz-Szab\'o~\cite{MR3694597}  to the general problem.

\smallskip

\noindent{\it Acknowledgements} This work has benefitted from discussions with Jim Hoste, Pat Gilmer and   Allison Moore; their help and insights are much appreciated.  In addition, Pat carried out the computations of Casson-Gordon invariants described in~Section~\ref{sec:hfcg}.

\section{Linking forms and metabolizers}
Let $M(J)$ denote the two-fold branched cover of $S^3$ branched over $J$.  If $J$ bounds an embedded surface $F\subset B^4$, let $W(F)$ denote the two-fold branched cover of $B^4$ branched over $F$.  In~\cite[Theorem 2]{MR656619}  and ~\cite{MR1690998}  a result is proved which immediately implies  the following.

\begin{theorem}\label{thm:linking}  If $J$ bounds a Mobius band $F \subset B^4$, then the linking form on $M(J)$     splits as $$   (H_1(M(J)), \text{lk}) \cong (G_1, \beta_1) \oplus (G_2, \beta_2)$$ where $G_1$ is cyclic (possibly trivial) and $\beta_2$ vanishes on  
$$\mo =  \text{im}\Big (\text{torsion}\big (H_2(W(F),M(J) ) \to H_1(M(J) ) \big) \Big) \subset G_2.$$ The order of $\mo$ satisfies  $|\mo |^2 = |G_2|$.  \end{theorem}
By definition, the form $(G_2, \beta_2)$ is called {\it metabolic} because it vanishes on a subgroup of order $\sqrt{|G_2|}$.

We now consider the case in which $J= K \cs K$ and  $H_1(M(K))$ is cyclic.  This includes two-bridge knots $K = B(s,q)$ for which $M(K)$ is the lens space $L(s,q)$.  Choose a prime divisor $p$ of $|H_1(M(K)| = s$ so that $H_1(M(K)) \cong \Z_{p^a} \oplus \Z_b$ for some $a >0$ and $\gcd(p,b) = 1$.  

Assume that $J = K \cs K$ bounds a ribbon Mobius band and restrict  the linking form to the {\it $p$--torsion subgroup} $\mathcal{H}_p \subset H_1(M(J))$ defined as the set of elements $x$ such that $p^kx = 0$ for some $k>0$.  Theorem~\ref{thm:linking} quickly implies  that there are two  possible splittings for the linking form on $\mathcal{H}_p$.

\begin{corollary}\label{thm:meta} Suppose  $J = K\cs K$ and $H_1(M(K))  \cong \Z_{p^a} \oplus \Z_b$, where $p$ does not divide $b$.  If $J$ bounds an embedded Mobius band in $B^4$, then the  linking form for $H_1(M(J))$ splits as one of two possibilities  given by Theorem~\ref{thm:linking}, with $\beta_2$ metabolic.  \begin{itemize}
\item $(\mathcal{H}_p, lk) \cong  (\Z_{p^a}, \beta_1) \oplus (\Z_{p^a}, \beta_2)$, or
\item $(\mathcal{H}_p, lk) \cong  (0, \beta_1) \oplus (\Z_{p^a}\oplus \Z_{p^a} , \beta_2).$
\end{itemize}

\end{corollary}
This  has the following consequence.

\begin{corollary}\label{cor:primes}  Suppose that $H_1(M(K)) \cong \Z_s$    and $K \cs K$ bounds a ribbon Mobius band in $B^4$.  If $p$ is a prime satisfying $p \equiv 3 \mod 4$, then  the exponent of $p$ in $s$ is even.   If  $p$ is a prime having odd exponent in $s$,  then $p \equiv 1\mod 4$ and   $(\mathcal{H}_p, lk) \cong  (0, \beta_1) \oplus (\Z_{p^a}\oplus \Z_{p^a} , \beta_2).$   \end{corollary}
\begin{proof}  If $a$ is odd, then a nonsingular form   on  $ \Z_{p^a}$ is not metabolic for any prime $p$.  If $p = 3 \mod 4$ and $a$ is odd, the form  $(\Z_{p^a}\oplus \Z_{p^a} , \beta_2)$    is a  direct sum  $(\Z_{p^a}, \beta_3) \oplus (\Z_{p^a},\beta_3)$ for some nonsingular form $\beta_3$.  Such a form cannot be  metabolic.  (The proof of this number theoretic  fact follows quickly from the theorem   that $-1$ is   a quadratic residue  module $p$ if and only if $p = 2$ or $ p  = 3\mod 4$.)
\end{proof}
\noindent As a quick application that we use later,   we have:
\begin{corollary}\label{cor:2-3}
If $H_1(M(K)) \cong \Z_s$ with  $s  \le 100$ and $K$ has a chiral smoothing, then 
$$ s \in \{5,9,13,17,25,29,37,41,45,49, 53, 61, 65, 73, 81, 85, 89, 97\}.$$
\end{corollary}

\section{Casson-Gordon invariants and nonorientable surfaces}

Let $M$ be closed three-manifold with $H_1(M, \Q) = 0$, and let 
$\rho \co H_1(M) \to \Z_m \subset \Q/\Z$.   Casson and Gordon defined an invariant $\sigma(M,\rho) \in \Q$.  This is additive over   connected sums and $\sigma(M, 0 ) = 0$. With regards to our work here, its key property is the following, proved in~\cite{MR2855790}.

\begin{theorem}\label{thm:cg}
Suppose that a closed three-manifold  $M = \partial W$, where $W$ is compact with $H_1(W, \Q) = 0$, and the inclusion $\pi_1(M) \to \pi_1(W)$ is surjective.  For each $\rho\co H_1(M) \to \Z_p$ that extends to $H_1(W)$, 
$$ |\sigma(M,\rho) | \le 2\beta_2(W) + 1 + \frac{1}{p-1}\beta_1(\widetilde{M}),$$ where $\widetilde{M}$ is the $p$--fold cover of $M$ associated to $\rho$.
\end{theorem}

To apply this, we consider branched covers over surfaces.  Suppose that $J$ bounds a Mobius band $F \subset B^4$.  Let $\overline{\rho} \in \mo$.  Linking with $\overline{\rho}$ defines a homomorphism $\rho \co H_1(M(J)) \to \Z_n$.  Since $\overline{\rho} \in \mo$, one can show that $\rho$ extends to $H_1(W(F))$.  If the order of $\rho$ in $\mo$ is $m$, then the nonsingularity of the linking form implies that the image of $\rho$ is cyclic of order $m$.   To simplify our discussion, we will henceforth assume that $m$ is a prime integer, which we will denote $p$.  
Theorem~\ref{thm:cg} has the following corollary.

\begin{corollary}\label{cor:mob} If $J$ bounds a ribbon Mobius band   $F \subset B^4$, then for all $\orho \in \mo$ of prime order $p$, 
$$ |\sigma(M(J),\rho) | \le 3 + \frac{1}{p-1}\beta_1(\widetilde{M}(J)).$$ 
\end{corollary}
\begin{proof}  Apply the theorem to $M = M(J)$ and $W = W(F)$.  Since $F$ is ribbon, the map $\pi_1(M) \to \pi_1(W)$ is surjective.  A result first proved by Massey~\cite{MR0250331} implies that $\beta_2(W(F)) = 1$.  
\end{proof}

In general, determining $\beta_1(\widetilde{M})$ might be difficult.  For lens spaces the value is easily computed.

\begin{theorem}\label{thm:betti1} Let $\rho \co H_1(L(s,q)) \oplus H_1(L(s,q)) \to \Z_p$ be a surjection.
\begin{itemize}
\item If $\rho$ is nontriivial on both summands, then  $\beta_1(\widetilde{M}(J)) = p-1$.

\item  If $\rho$ is trivial on one of the two summands, then  $\beta_1(\widetilde{M}(J)) =0$.
\end{itemize}

\end{theorem}


\section{Two-bridge knots and Casson-Gordon invariants of lens spaces}\label{sec:cg-obs}

We begin by noting  that Corollary~\ref{thm:meta} implies the following.

\begin{corollary}  If $K$ is a two-bridge knot $B(s,q)$  and  $J = K \cs K$ bounds a ribbon Mobius band $ F \subset B^4$, then for every prime divisor of $p$ of $s$   there is an element $\orho \in \mo$ of order exactly $p$.  Consequently, there is a surjective homomorphism $\rho\co H_1(M(J)) \to \Z_p$ that extends to $H_1(W(F))$.  \end{corollary}

We will consider surjective  homomorphisms $\rho \co \pi_1(L(s,q)) \to \Z_p$, so will  write $s = pn$. (In our main reference,~\cite{MR900252}, the letter $m$ is used instead of $p$, and $m$ was not assumed to be prime.  Stronger results could be obtained here by not restricting to the prime setting, but that generality is not required to generate interesting families of examples.)

In~\cite{MR900252},    the values of $\sigma(L(s,q), \rho)$ are presented  within a computation preceding the unnumbered corollary~\cite[page 188]{MR900252}.  We restate that result below as Theorem~\ref{thm:comp}.  There is a subtlety to the formula that appears there, as   it depends on the choice of a particular generator of $H_1(L(s,q))$.  We will not specify that choice here, but note that since we will be considering the set of all values, knowing the choice is not necessary for computation.  That is, we are interested in the set of all values that arise for $\orho$ and its multiples, so can avoid that technicality.  Also, since $\sigma(L(s,q), \rho) = \sigma(L(s,q), -\rho) $ we can further restrict the set of values considered.

\begin{theorem}\label{thm:comp} Let $\orho$ be an element of order $p$ in $H_1(L(pn,q))$.  Then for $0<r< p$, 
$$\sigma(L(pn,q), r q  \rho) =   4\left(\text{area } \Delta(nr,\frac{qr}{p}) - \text{int } \Delta(nr,\frac{qr}{p}) \right)  . $$ 
\end{theorem}
Here $\Delta(x,y)$ represents a triangle with vertices $\{(0,0), (x,0), (0,y)\}$.  The value of $\text{int } \Delta(x,y )$ is the number of integer lattice points in   the triangle, with lattice points on the interior of edges counting $1/2$,  non-zero vertices in the integer lattice count $1/4$, and the vertex at the origin is not counted.

Focusing on the case of $L(s,1)$, this gives the following.

\begin{corollary} Let $p$ be a prime factor of $s = pn$, and let $\rho$ be an element of $\mo$ of order $p$.  Then for all $r$, $0 < r <p$, 
\begin{equation*}
\begin{split}
\sigma(L(s,1), r \rho)&=4\left(\text{area } \Delta(nr,\frac{r}{p}) - \text{int } \Delta(nr,\frac{r}{p})\right)\\
&= \frac{2n}{p}r^2 - 2n r +1.
\end{split}
\end{equation*}
\end{corollary}

\begin{proof} The computation of the lattice point count is simplified by the fact that $r/p <1$, so that the only points in the count are on the left edge of the triangle.  The rest of the computation is simple algebra.
\end{proof}

\begin{corollary}\label{cor:maxmin} The maximum and minimum values of $-\sigma(L(pn,1), r \rho)$ for $0<r<p-1$ are:
\begin{itemize}
\item  Minimum   $= 2n(1- \frac{1}{p} )  -1> 0$;  occurs at $r = 1$ and $r = p-1$.
\item Maximum    $=  \frac{n}{2}(p -\frac{1}{p} ) -1$; occurs at $ r = \frac{p \pm 1}{2})$.
\end{itemize}
\end{corollary}


\section{Torus knots $T(2,m)$.}
We can now use the results of Theorem~\ref{thm:cg} and Corollary~\ref{cor:maxmin} to restate the constraint from Corollary~\ref{cor:mob}.  This is sufficient to rule out chiral smoothings for torus knot $T(2,m) = B(m,1)$ for all odd $m$ other than $m = 5$ and $m=9$.  As illustrated in the introduction, $T(2,5)$ does support a chiral smoothing. The case of $T(2,9)$ is unknown. 

Let $K = T(2,m)$ and suppose that $K$ admits a chiral smoothing; in particular, assume that  $K \cs K$ bounds a ribbon Mobius band   $F \subset B^4$.  Note that  $M(K \cs K) = L(m,1) \cs L(m,1)$.   Let $m = pn$ for some odd prime $p$, and let $\orho$ be an element of order $p$ in $\mo$.   There are two cases to consider. 
\smallskip

\noindent{\bf Case I.} If $\rho$ is nontrivial on both natural summands of $H_1(L(m,1) \cs L(m,1))$ then
$$\left( 2n(1- \frac{1}{p} )  -1\right) + \left(  \frac{n}{2}(p -\frac{1}{p} ) -1\right) \le 4.$$
To see this, we observe that choosing the correct multiple of $\rho$, we can assure that one of the two values of the Casson-Gordon invariant is at the maximum.

\noindent{\bf Case II.} If $\rho$ is trivial on one of the  natural summands of $H_1(L(m,1) \cs L(m,1))$ then
$$\left(  \frac{n}{2}(p -\frac{1}{p} ) -1\right) \le 3.$$
Again, this follows by  choosing the   multiple of $\rho$ that which the  Casson-Gordon invariant is at its maximum.

With these two bounds, the proof of the following theorem is immediate.

\begin{theorem} If the knot $K = T(2,m), m>1$,  admits a chiral smoothing, then $m= 5$ or $m=9$.
\end{theorem}

\begin{proof}
If $p \ge 11$, then $\frac{1}{2}(p - \frac{1}{p}) -1 >4$, so the inequality is violated regardless of $n$.  By Corollary~\ref{cor:2-3}, the only remaining possibilities are $m=5$ and $m=9$.

\end{proof}

It is interesting to observe that in the one unknown case, $T(2,9)$, we would consider $p=3$ and $n=3$.  There  are then   two inequalities to consider, both of which can be seen to actually be equalities:
\begin{itemize}

\item $\left( 3(1- \frac{1}{3} )  -1\right) + \left(  \frac{3}{2}(3 -\frac{1}{3} ) -1\right) = 4.$

\item $\left(  \frac{3}{2}(3 -\frac{1}{3} ) -1\right)=  3.$

\end{itemize}

\section{Further metabolizer constraints}

\subsection{Identifying metabolizing vectors} In the following discussion, we use the identification of $H_1(M)$ with $\Q/\Z$--valued  characters on $H_1(M)$ that arises from the linking form. 
To generalize the examples of the previous section, we observe that   Corollary~\ref{cor:primes}  leads to the following result. \begin{theorem}
Suppose that $K$ supports a  chiral smoothing, that   $H_1(M(K)) \cong \Z_s$, and that $p$ is a prime divisor of $s$ with odd exponent $a$.  Let $\rho$ be  a nontrivial $\Z_p$--valued character on $H_1(M(K))$.  Then for some $\alpha \in \Z_p  $ with $1+\alpha^2 = 0 \mod p$ and for all $r$:
$$\big| \sigma(K, r\rho  ) + \sigma(K, r\alpha \rho)\big| \le 4.$$

\end{theorem}

\begin{proof} Let $a = 2k +1$.  According to Corollary~\ref{cor:primes}, if  $H_1(M(K))_p \cong \Z_{p^a}$, then $H_1(M(K\cs K))_p  \cong  \Z_{p^a} \oplus \Z_{p^a}$ has a metabolizer  $\mo$ of order $p^a$, and thus $\mo \cong \Z_{p^b} \oplus \Z_{p^c}$, where $b + c = a$.  We can assume that $b \ge k +1$.   In particular, $\mo$ contains an element of order at least $p^{k+1}$.  Some multiple of this element is of order exactly $p^{k+1}$.  By taking a further multiple, we see that $\mo$ contains an element $g = (p^k, \alpha p^j) \in Z_{p^a} \oplus \Z_{p^a}$, for some $j \ge k$, where $\gcd(\alpha,p) =1$.    This element has self-linking 0 if and only if $j = k$ and $1 + \alpha^2 $ is divisible by $p$.   

Suppose that  the self-linking of $(1, 0) \in  H_1(M(K\cs K))_p $ is $\gamma/p^a$, where $\gcd(\gamma,p) = 1$.  If  $g$ is  multiplied by $p^k$, we get the element $   (p^{a-1}, \alpha p^{a-1}) \in \mo$,  a metabolizing element of order $p$ that takes  value $\gamma$ on the generator of the first summand  and $\alpha \gamma$ on the second.  
The pairs of metabolizing characters, $(r\rho,   \alpha r \rho)$ all arise as multiples of this element.
\end{proof}

\begin{example}

Consider the two-bridge knot $K = B(17,2)$, also known as $10_1$, the four-twisted double of the unknot.  We have $H_1(M(K)) \cong \Z_{17}$, and a calculation using Theorem~\ref{thm:comp} yields the following values of Casson-Gordon invariants, listed as $(r, \sigma(M(K), 2r\rho))$
$$\{ (1, -\frac{13}{17}), (2, -\frac{35}{17}), (3, -\frac{49}{17}), (4, -\frac{55}{17}), (5, -\frac{53}{17}), (6,  -\frac{43}{17}),  (7, -\frac{25}{17}), (8,  \frac{1}{17})    \}.$$

For any nonsingular linking form on $\Z_{17}$, the metabolizer for the direct sum, $\Z_{17}\oplus \Z_{17}$ consists of the multiples of $(1,\pm 4)$.  Thus, the metabolizer contains $(3,\pm 5)$ and the absolute value of the sum of the two Casson-Gordon invariants is $102/17 > 4$.  An obstruction to chirally smoothing this knot could not be derived by considering the maximum and minimum values of the Casson-Gordon invariants.

\end{example}

\subsection{Non-prime order metabolizing elements}

We have concentrated on the case of characters taking values in $\Z_p$ for some prime $p$.   In the setting of ribbon surfaces, Theorem~\ref{thm:cg} applies without the constraint the $p$ be prime (see~\cite{MR2855790} for details).  However, Theorem~\ref{thm:betti1} does require that $p$ be prime.  Computing $\beta_1(\widetilde{M})$ for connected sums of lens spaces is an elementary exercise in covering spaces and the Euler characteristic, using the fact that any connected cover of a lens space is again a lens space. In all cases, the quotient $\frac{1}{p-1}\beta_1( \widetilde{M}) \le 1$.

\begin{example}  To find a knot in which the basic constraints do not apply, we need to consider a two-bridge knot $B(s,q)$, where $s$ is composite and all prime factors that equal $3$ modulo $ 4$ have even exponent.  Furthermore, we have already handled the case of $q = 1$, so must consider a larger value of $q$.  A basic example is $K= B(3^2 \cdot13, 20)$. In this case,   if we consider only characters to $\Z_3$, the absolute value of the sum of the maximum and minimum Casson-Gordon invariant is $2$.  For characters to $\Z_{13}$ the sum is $4$.  However, for surjective characters to $\Z_{39}$ the sum is $\frac{53}{13} >4$.  Thus, this knot is obstructed from supporting a chiral smoothing.
\end{example}

\subsection{Doubled knots}
We have restricted to two-bridge knots, largely because this led to settings in which Casson-Gordon invariants are readily computed.  In the case of doubled knots, the invariants can often be computed using results of Gilmer~\cite{MR711523}.   

It was noticed by Kanenobu~\cite{MR711523} that a so-called $4$--move can be performed on a knot via a single band move.  In particular, a single smoothing in a diagram for the positive clasped, $k$--twisted double of a knot $K$,  $D_+(K,t)$ converts it to the negative clasped double, $D_-(K,t)$.   In general,   $-D_+(K,t) = D_-(-K,-t)$;  thus,  if $K$ is amphicheiral, a single smoothing converts $D_+(K,0)$ into $-D_+(K,0)$, so these doubled knots support chiral smoothings.  What if $t \ne 0$?

\begin{example} The figure eight knot, $4_1$, can be described as $D_+(U, -1)$ where $U$ is the unknot.  We have $M(4_1) = L(5,2)$.  More generally, for any  $K$, $H_1(M(D_+(K,-1))) \cong \Z_5$.  According to Gilmer's work, for any surjective character $\rho\co  H_1(M(D_+(K,-1))) \to \Z_5$,   the value of $\sigma( M(D_+(K,-1)), \rho)$ is determined by the values of $\sigma( L(5,2), \rho')$ (for some nontrivial $\rho'$) and the classical Levine-Tristram signatures~\cite{MR0253348,MR0248854} of $K$.  We state the formula without specifying the relevant nontrivial character $\rho'$ or the value of $i$, noting only that $i \ne 0 \mod 5$.
$$   \sigma(M(D_+(K,-1)),\rho) = \sigma(L(5,2),\rho') + 2\sigma_{i/5}(K).$$  The values of $\sigma(L(5,2), \rho)$ for nontrivial $\rho$ are $\pm 1/5$.  Calculations such as those done earlier in this paper yield the following, which applies, for instance, for $K$ any nontrivial torus knot.

\begin{theorem}If $|\sigma_{1/5}(K) + \sigma_{2/5}(K) |\ge 2$, then $D_+(K,-1)$ does not support a chiral smoothing.  \end{theorem}
\end{example}

\subsection{Smoothing distance} Rather than ask if a single smoothing can covert $K$ into $-K$, one can more generally ask  whether a single smoothing can convert a knot $K$ into a knot $J$.  

\begin{example} The first knot that was shown to be algebraically slice but not slice by Casson and Gordon~\cite{MR900252} was the two-bridge knot $B(25,2)$.  Since the appearance of a prime of even power in the first homology introduces challenges, we consider a related family of examples,  the set of four  two-bridge knots $\{B(25,1) , B(25,24), B(25,2), B(25,23)\}$.  The first two are mirror images, as are the last two.   
For each, up to conjugation there are two characters to $\Z_5$ and the pair of  values of the Casson-Gordon invariants is, for each knot, $\{(-7,-11), (7,11), (-3,-5), (3,5)\}$.

In general, if there is a smoothing that converts  a knot $K$ into a knot $J$, then $K \cs -J$ bounds a Mobius band in the four-ball.  The bounds based
on Theorem~\ref{thm:cg} continue to apply.  Using the additivity of the Casson-Gordon invariants, it is then clear that no two of these four knots differ by a single smoothing.

This example cannot be expanded to include all knots $B(25,q)$.  For instance,  consider the knot $B(25,8)$.  This knot has a diagram corresponding to each continued fraction expansion of $25/8$.   Consider the particular expansion  $$
\frac{25}{8}=4 + \cfrac{1}{-2+\cfrac{1}{2 +\cfrac{1}{-1 + \cfrac{1}{-5}}}} 
$$
Smoothing a single crossing yields a two-bridge knot with diagram corresponding to the continued fraction
$$
\frac{25}{7}=4 + \cfrac{1}{-2+\cfrac{1}{2 +\cfrac{1}{0 + \cfrac{1}{-5}}}}
$$  
Thus, the two-bridge knots $B(25,7)$ and $B(25,8)$ differ by a single smoothing.  These are the knots $8_9$ and $-11a{364}$ in the standard notation~\cite{KnotInfo}.

\end{example}

\section{Heegaard Floer obstructions and further examples of torus knots}\label{sec:hfcg}
In~\cite{MR3694597}, Ozsv\'ath-Stipcisz-Szab\'o developed  a bound on the nonorientable four-ball genus of a knot $K$ in terms of the {\it little upsilon function}, $\upsilon(K)$, and the classical signature~\cite{MR0171275},  $\sigma(K) $.   In brief, if $K$ bounds a connected surface $F\subset B^4$, then $\beta_1(F) \ge | \upsilon(K) - \sigma(K)/2|$.  Both $\upsilon$ and $\sigma$ are additive functions, so for knots of the form $K\cs K$ the value of this difference is even, and in particular cannot equal 1.  Thus, we have the following theorem.

\begin{theorem} If $\upsilon(K) \ne \sigma(K)/2$, then $K$ does not support a chiral smoothing.
\end{theorem}
 Denote the difference $\upsilon(K) -  \sigma(K)/2$ by $\phi(K)$. 

For alternating knots $K$, $\phi(K) = 0$, so no obstruction arises.  For  general torus knots, recursive formulas are available:  for the signature function, see~\cite{MR617628}, and for the upsilon function, see~\cite{MR3694648}.  For a more general discussion, see~\cite{2018arXiv180901779J}.   For instance, one can show that   $\phi(T(3,4+6k)) = 1$ and  $T(3,5+6k)=1$ for all $k\ge 0$, and thus these admit no chiral smoothings.  On the other hand,   $\phi(T(3,7)) = \phi(T(3,8)) =0$ and we do not know if these knots admit chiral smoothings.   For $T(3,8)$, the two-fold cover has non-trivial first homology, isomorphic to $\Z_3$, so conceivably Casson-Gordon theory could be applied, but those computations have not been undertaken.

As another example, for positive torus knots of the form $T(4,4k+7)$ we have $\phi(T(4, 4k +7)) = 1$, so there is an obstruction to these knots admitting chiral smoothings.  On the other hand, for all positive $k$,  $\phi(T(4,4k+5)) =0$, and there is no obstruction of this form.  However, in these cases, Casson-Gordon theory can be applied.  For the first possible example,  $H_1(M_2(T(4,5) \cs T(4,5))) \cong  (\Z_5)^2$, but the obstructions similar to those developed in Section~\ref{sec:cg-obs} vanish.  Work of Gilmer  and Orevkov~\cite{2015arXiv151006377G}  facilitate the computation of the Casson-Gordon  invariants of the two-fold branched covers of torus knots.  Using results there, along with programs written by Gilmer, torus knots of the form $T(4,4k +5)$ have been shown to not support chiral smoothings  for $7 \le k \le 50$.  A proof that the obstruction holds for all $k>50$ has not be achieved.
More generally, all computational evidence points to it being true that for all $m$ and for all large $n$ with $\gcd(2m ,n) = 1$, the torus knot $T(2m, n)$ does  not support a chiral smoothing.



\begin{thebibliography}{10}

\bibitem{MR3307286}
T.~Abe and T.~Kanenobu.
\newblock Unoriented band surgery on knots and links.
\newblock {\em Kobe J. Math.}, 31(1-2):21--44, 2014.

\bibitem{MR900252}
A.~J. Casson and C.~M. Gordon.
\newblock Cobordism of classical knots.
\newblock In {\em \`A la recherche de la topologie perdue}, volume~62 of {\em
  Progr. Math.}, pages 181--199. Birkh\"auser Boston, Boston, MA, 1986.
\newblock With an appendix by P. M. Gilmer.

\bibitem{KnotInfo}
J.~C. Cha and C.~Livingston.
\newblock Knotinfo: Table of knot invariants, {\it
  http://www.indiana.edu/$_{\widetilde{~}}$knotinfo}.
\newblock August 14, 2018.

\bibitem{MR3694648}
P.~Feller and D.~Krcatovich.
\newblock On cobordisms between knots, braid index, and the upsilon-invariant.
\newblock {\em Math. Ann.}, 369(1-2):301--329, 2017.

\bibitem{MR656619}
P.~M. Gilmer.
\newblock On the slice genus of knots.
\newblock {\em Invent. Math.}, 66(2):191--197, 1982.

\bibitem{MR711523}
P.~M. Gilmer.
\newblock Slice knots in {$S^{3}$}.
\newblock {\em Quart. J. Math. Oxford Ser. (2)}, 34(135):305--322, 1983.

\bibitem{MR2855790}
P.~M. Gilmer and C.~Livingston.
\newblock The nonorientable 4-genus of knots.
\newblock {\em J. Lond. Math. Soc. (2)}, 84(3):559--577, 2011.

\bibitem{2015arXiv151006377G}
P.~M. {Gilmer} and S.~Y. {Orevkov}.
\newblock {Signatures of real algebraic curves via plumbing diagrams}.
\newblock {\em ArXiv e-prints}, Oct. 2015.

\bibitem{MR617628}
C.~M. Gordon, R.~A. Litherland, and K.~Murasugi.
\newblock Signatures of covering links.
\newblock {\em Canad. J. Math.}, 33(2):381--394, 1981.

\bibitem{MR1075165}
J.~Hoste, Y.~Nakanishi, and K.~Taniyama.
\newblock Unknotting operations involving trivial tangles.
\newblock {\em Osaka J. Math.}, 27(3):555--566, 1990.

\bibitem{2017arXiv170700152I}
K.~{Ichihara}, T.~{Ito}, and T.~{Saito}.
\newblock {Chirally cosmetic surgeries and Casson invariants}.
\newblock {\em ArXiv e-prints}, July 2017.

\bibitem{2016arXiv160201542I}
K.~{Ichihara}, I.~D. {Jong}, and H.~{Masai}.
\newblock {Cosmetic banding on knots and links}.
\newblock {\em ArXiv e-prints}, Feb. 2016.

\bibitem{2018arXiv180901779J}
S.~{Jabuka} and C.~A. {Van Cott}.
\newblock {On a Nonorientable Analogue of the Milnor Conjecture}.
\newblock {\em ArXiv e-prints}, Sept. 2018.

\bibitem{MR2755489}
T.~Kanenobu.
\newblock Band surgery on knots and links.
\newblock {\em J. Knot Theory Ramifications}, 19(12):1535--1547, 2010.

\bibitem{MR2812265}
T.~Kanenobu.
\newblock {$H(2)$}-{G}ordian distance of knots.
\newblock {\em J. Knot Theory Ramifications}, 20(6):813--835, 2011.

\bibitem{MR3548474}
T.~Kanenobu.
\newblock Band surgery on knots and links, {III}.
\newblock {\em J. Knot Theory Ramifications}, 25(10):1650056, 12, 2016.

\bibitem{MR2573402}
T.~Kanenobu and Y.~Miyazawa.
\newblock {$H(2)$}-unknotting number of a knot.
\newblock {\em Commun. Math. Res.}, 25(5):433--460, 2009.

\bibitem{MR0253348}
J.~Levine.
\newblock Invariants of knot cobordism.
\newblock {\em Invent. Math. 8 (1969), 98--110; addendum, ibid.}, 8:355, 1969.

\bibitem{MR859958}
W.~B.~R. Lickorish.
\newblock Unknotting by adding a twisted band.
\newblock {\em Bull. London Math. Soc.}, 18(6):613--615, 1986.

\bibitem{MR0250331}
W.~S. Massey.
\newblock Proof of a conjecture of {W}hitney.
\newblock {\em Pacific J. Math.}, 31:143--156, 1969.

\bibitem{2018arXiv180602440M}
A.~H. {Moore} and M.~{Vazquez}.
\newblock {A note on band surgery and the signature of a knot}.
\newblock {\em ArXiv e-prints}, June 2018.

\bibitem{MR1690998}
H.~Murakami and A.~Yasuhara.
\newblock Four-genus and four-dimensional clasp number of a knot.
\newblock {\em Proc. Amer. Math. Soc.}, 128(12):3693--3699, 2000.

\bibitem{MR0171275}
K.~Murasugi.
\newblock On a certain numerical invariant of link types.
\newblock {\em Trans. Amer. Math. Soc.}, 117:387--422, 1965.

\bibitem{MR3694597}
P.~S. Ozsv\'ath, A.~I. Stipsicz, and Z.~Szab\'o.
\newblock Unoriented knot {F}loer homology and the unoriented four-ball genus.
\newblock {\em Int. Math. Res. Not. IMRN}, (17):5137--5181, 2017.

\bibitem{MR0248854}
A.~G. Tristram.
\newblock Some cobordism invariants for links.
\newblock {\em Proc. Cambridge Philos. Soc.}, 66:251--264, 1969.

\bibitem{MR3331990}
A.~Zekovi\'c.
\newblock Computation of {G}ordian distances and {$H_2$}-{G}ordian distances of
  knots.
\newblock {\em Yugosl. J. Oper. Res.}, 25(1):133--152, 2015.

\end{thebibliography}

\end{document}